\documentclass[12pt]{article}

\usepackage{amsmath,amsthm,amsfonts,amssymb,amscd,cite}
\usepackage{graphicx}

\allowdisplaybreaks[4]

\newtheorem {lemma}{Lemma}
\newtheorem {theorem} {Theorem}

\newtheorem{claim}{Claim}
\allowdisplaybreaks [4]

\begin{document}

\title{\bf Toughness and existence of $2$-factors}

\author{Leyou Xu\footnote{email: leyouxu@m.scnu.edu.cn}, Bo Zhou\footnote{Corresponding author; email: zhoubo@m.scnu.edu.cn}\\
School of Mathematical Sciences, South China Normal University\\
Guangzhou 510631, P.R. China}

\date{}
\maketitle

\begin{abstract}
A graph is $t$-tough if the deletion of any set of, say, $m$ vertices from the graph leaves a graph with at most $\frac{m}{t}$ components.  In 1973, Chv\'{a}tal  suggested the problem of relating toughness to factors in graphs. In 1985, Enomoto et al. showed that  each $2$-tough graph with at least three vertices has a $2$-factor, but for any $\epsilon>0$, there exists a $(2-\epsilon)$-tough graph on at least $3$ vertices having no $2$-factor.
In recent years, the study of sufficient conditions
for graphs with toughness less than $2$ having a $2$-factor has received a paramount interest.
In this paper, we give new tight sufficient conditions for a $t$-tough graph having a $2$-factor when $1\le t<2$ by involving  independence number,  minimum degree, connectivity and forbidden forests, improving and extending some existing results.

\end{abstract}

\section{Introduction}

Let $G$ be a graph with vertex set $V(G)$ and edge set $E(G)$.
For $v\in V(G)$,  $N_G(v)$ denotes the neighborhood of $v$ in $G$ and $d_G(v)$  the degree of $v$, i.e., $d_G(v)=|N_G(v)|$. Denote by $\delta(G)$ the minimum degree of $G$.
For a nonnegative integer $k$, a graph is called $k$-regular if $d_G(v)=k$ for each $v\in V(G)$.
A $k$-factor of a graph
is a $k$-regular spanning subgraph of the graph. A Hamilton cycle is a connected $2$-factor.
The independence number $\alpha(G)$ of $G$ is the cardinality of a maximum set of independent
vertices of $G$.
A connected graph $G$ is  $k$-connected if it has more than $k$ vertices and remains connected whenever fewer than $k$ vertices are deleted.

For a given graph $H$, a graph $G$ is called $H$-free if $G$ does not contain $H$ as an induced subgraph.
For vertex disjoint graphs $H$ and $F$, $H\cup F$ denotes the disjoint union of graphs $H$ and $F$.
For positive integer $a$ and a graph  $H$, $aH$ denotes the graph consisting of $a$ disjoint copies of $H$. As usual, $P_n$ denotes the path on $n$ vertices.
A forest is linear if it consists of paths.

For a graph $G$ with $\emptyset\ne S\subset V(G)$, denote by $G[S]$ the subgraph of $G$ induced by $S$.
Let $G-S=G[V(G)\setminus S]$.
The number of components of $G$ is denoted by $c(G)$.
For a non-complete connected graph $G$, a cut set  is a subset $S\subset V(G)$ such that $c(G-S)>1$.
The toughness of a graph $G$, denoted by $\tau(G)$, is defined as
\[
\tau(G)=\min\frac{|S|}{c(G-S)}
\]
where the minimum is taken over all vertex cuts $S$ of $G$  if $G$ is non-complete
$\tau(G)=\infty$ otherwise.
For a positive real number $t$,
a graph $G$ is said to be $t$-tough if $\tau(G)\ge t$.
This concept was introduced by  Chv\'{a}tal \cite{Chv}, who
conjectured that there exists a constant $t_0$ such that any $t_0$-tough graph has a Hamiltonian cycle,  and that any $k$-tough graph on $n$ vertices with $n\ge k+1$ and $kn$ even has a $k$-factor.
Enomoto,
Jackson, Katerinis and Saito \cite{EJKS} confirmed the latter and proved  that the
 result is best possible in the sense that `$k$-tough' cannot be replaced by `$(k-\epsilon)$-tough' for any positive real number $\epsilon$. Thus,
$t_0\ge 2$, and for any $\epsilon>0$, there exists a $(2-\epsilon)$-tough graph on at least $3$ vertices having no $2$-factor.

So, in recent years, the study of sufficient conditions
for graphs with toughness less than $2$ having a $2$-factor has received a paramount interest.

\begin{theorem}[Bauer and Schmeichel] \label{B1}
Let $t$ be a rational number with $1\le t<2$, and let $G$ be a $t$-tough graph on at least three vertices. 
\begin{enumerate}
\item[(i)] \cite[Theorem 3]{BS} If  $\delta(G)\ge
\frac{(2-t)n}{1+t}$,   then $G$ contains a $2$-factor.

\item[(ii)] \cite[Theorem 5]{BS} If $\frac{3}{2}\le t<2$ and $\delta(G)\ge \frac{(3t-2-t^2)n}{7t-7-t^2}$, then $G$ contains a $2$-factor.
\end{enumerate}
\end{theorem}

The result in Theorem \ref{B1} is asymptotically best possible for every rational $t\in [1,\frac{3}{2})$ and for infinitely many rational $t\in [\frac{3}{2},2)$ \cite{BS}.

Bauer et al. \cite{BKKV} proved that  every $\frac{3}{2}$-tough $5$-chordal graph  with at least three vertices
has a $2$-factor \cite{BKKV}, where a $5$-chordal graph is one with no induced cycles of length larger than $5$. This implies that every $\frac{3}{2}$-tough $2P_2$-free graph on at least three vertices has a $2$-factor, see \cite{OS1}.
Sanka  \cite{San} proved that every $\frac{3}{2}$-tough $(P_{10}\cup P_4)$-free graph  has a $2$-factor.
Grimm, Johnsen and Shan \cite{GJS} found sharp  bound less than $2$ on $t$ such that every $t$-tough $R$-free graphs with at least three vertices has a $2$-factor for any forest $R$ on $5,6,7$ vertices.
There are also works on
characterization  all pairs of graphs R, S such that every $R$ and $S$-free graph of sufficiently large order has a $2$-factor, see, e.g. \cite{PRP}.
The connectivity is used to study the
existence of Hamilton cycle in  $R$-free graphs for a linear forest $R$ in \cite{SS,XLZ}.

Inspired the above work, in this paper, we give new sufficient conditions for $t$-tough graphs with $t<2$ to have a $2$-factor. The main results are the following theorems.
The first one involves independence number and minimum degree.
The second and third ones
involve  connectivity and forbidden forests.

\begin{theorem}\label{inde2}
Let $\epsilon$ be a rational number, $0<\epsilon \le 1$.
Let $G$ be a $(2-\epsilon)$-tough graph with at least three vertices. If $\delta(G)\ge \epsilon\alpha(G)$, then $G$ contains a $2$-factor.
\end{theorem}

Theorem \ref{inde2} echoes the classical result of Chv\'{a}tal and Erd\H{o}s \cite{CE} stating that
every graph $G$ on at least three vertices with $\kappa(G)\ge \alpha(G)$ (known as  Chv\'{a}tal--Ed\H{o}s condition)  has a Hamilton cycle, and the result of  Niessen \cite{Nis}: every graph $G$ on at least three vertices  with  $\delta(G)>\alpha(G)$ has a $2$-factor. Let $G$ be the graph in Theorem \ref{inde2} with 
$n=|V(G)|$ and $t=2-\epsilon$. If $\alpha(G)=1$, then $G\cong K_n$. Suppose that $\alpha(G)\ge 2$. Then $t\le \tau(G)\le \frac{n-\alpha(G)}{\alpha(G)}$, so $\alpha(G)\le \frac{n}{1+t}$, which implies that $\frac{(2-t)n}{1+t}\ge \epsilon\alpha(G)$.  Thus, Theorem \ref{inde2} improves \cite[Theorem 3]{BS}.



\begin{theorem}\label{t1} The following statements are true.
\begin{enumerate}
\item[(i)]
Let $\ell=1,2$ and let $k$ be a positive integer. Any $1$-tough $(k+\ell-1)$-connected
$(P_{2\ell}\cup kP_1)$-free graph with at least three vertices has a $2$-factor.

\item[(ii)]
Let $k$ be a positive integer.
Any $1$-tough $(k+1)$-connected $(P_3\cup kP_1)$-free graph  with at least three vertices has a $2$-factor.
\end{enumerate}
\end{theorem}

\begin{theorem}\label{p5} The following statements are true.
\begin{enumerate}
\item[(i)]
Let $\ell=2,3$ and let $k$ be a positive integer.
Any $\frac{3}{2}$-tough $(k+\ell-1)$-connected $(P_{2\ell+1}\cup kP_1)$-free graph  with at least three vertices has a $2$-factor.

\item[(ii)]
Let $k$ be a positive integer.
Any $\frac{3}{2}$-tough $(k+2)$-connected $(P_6\cup kP_1)$-free graph with at least three vertices has a $2$-factor.
\end{enumerate}
\end{theorem}

Denote by $K_n$ the complete graph on $n$ vertices, and $\overline{K_n}$ the complement of $K_n$. For two disjoint graphs
$G$ and $H$, the join of $G$ and $H$, denoted by $G\vee H$, is formed from $G\cup H$ by adding all possible edges between vertices of $G$ and vertices of $H$.

\begin{figure}[htbp]
\centering
\includegraphics[width=0.6\textwidth]{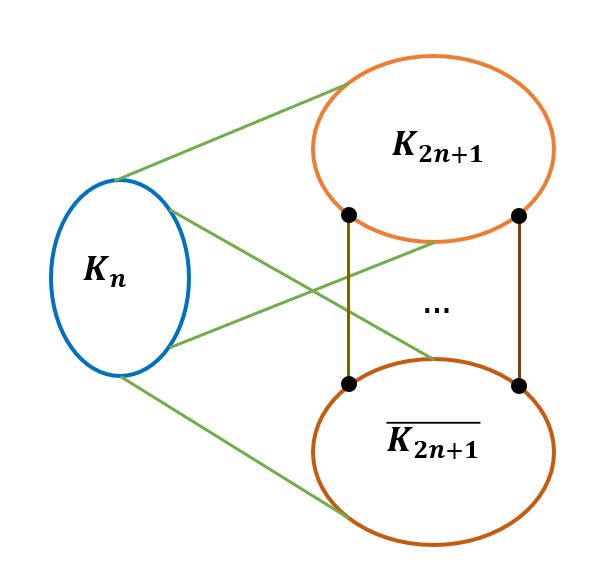}
\caption{The graph $H_{n}$}
\label{fh1}
\end{figure}

\noindent
{\bf Remark 1.}  
(a) For positive integer $n$, let $H_{n}$
be a graph obtained from $K_{n}\vee (\overline{K_{2n+1}}\cup K_{2n+1})$ by adding a perfect matching between vertices in $\overline{K_{2n+1}}$ and vertices in $K_{2n+1}$, see Fig.  \ref{fh1}.
It is easy to see that $\alpha(H_n)=2n+1$ and $\delta(H_n)=n+1$. From the proof of Theorem 3.3 in \cite{Chv}, we have  $\tau(H_n)=\frac{3n}{2n+1}$ and $H_{n}$ has no $2$-factor.
As $H_n$ is $(2-\frac{n+2}{2n+1})$-tough with $\delta(H_n)<n+2=\frac{n+2}{2n+1}\alpha(H_n)$,  the minimum degree condition in Theorem \ref{inde2} is sharp.
As $H_n$ is a graph with $\delta(H_n)\ge\frac{n+1}{2n+1}\alpha(H_n)$ and $\tau(H_n)=2-\frac{n+2}{2n+1}$, the toughness condition in Theorem \ref{inde2} is  asymptotically sharp.

(b) There are infinite graphs to which Theorem \ref{inde2}   applies  but Theorem 5 in \cite{BS} does not.

For positive integers $m$, $a$, $b$ and $c$, let $R_{m,a,b,c}=K_{cm}\vee amK_{bm}$.
Note that $\delta(R_{m,a,b,c})=(b+c)m-1$, $\alpha(R_{m,a,b,c})=am$ and $\tau(R_{m,a,b,c})=\frac{cm}{am}=\frac{c}{a}$.


Let $m\ge \frac{5(7ac-7a^2-c^2)}{2(3ac-2a^2-c^2)}$ and $\frac{3}{2}a\le c<2a$.
As
\[
\delta(R_{m,a,b,c})=(b+c)m-1>(2a-c)m=(2-\tau(R_{m,a,b,c}))\alpha(R_{m,a,b,c}),\]
we have by Theorem \ref{inde2} that $R_{m,a,b,c}$ has a $2$-factor.
However, the condition of \cite[Theorem 5]{BS} is not satisfied. This is because
as $c\ge \frac{3}{2}a$, we have
\[
m\ge \frac{5(7ac-7a^2-c^2)}{2(3ac-2a^2-c^2)}\ge \frac{b+c}{ab}\cdot \frac{7ac-7a^2-c^2}{(3ac-2a^2-c^2)},
\]
so
\[
\frac{3ac-2a^2-c^2}{7ac-7a^2-c^2}abm\ge  b+c,
\]
implying that 
\[
\frac{3\tau(R_{m,a,b,c})-2-\tau(R_{m,a,b,c})^2}{7\tau(R_{m,a,b,c})-7-\tau(R_{m,a,b,c})^2}(ambm+cm)>(b+c)m-1=\delta(R_{m,a,b,c}).
\]

\begin{figure}[htbp]
\centering
\begin{minipage}{0.48\textwidth}
\centering
\includegraphics[width=0.97\textwidth]{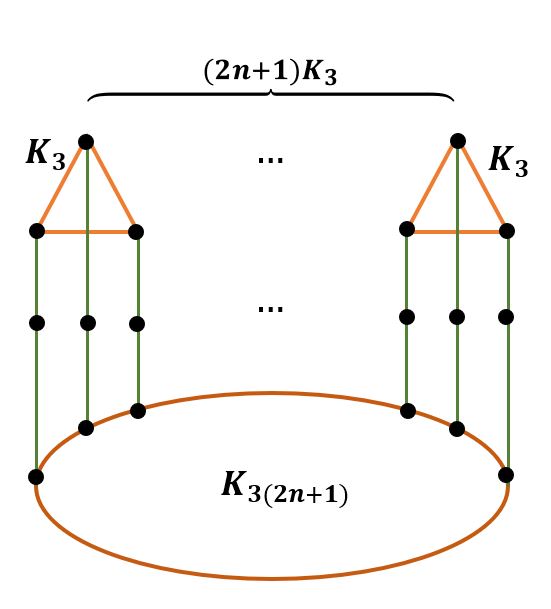}
\caption{The graph $G'_{n,k}$}
\label{g1}
\end{minipage}
\begin{minipage}{0.48\textwidth}
\centering
\includegraphics[width=0.98\textwidth]{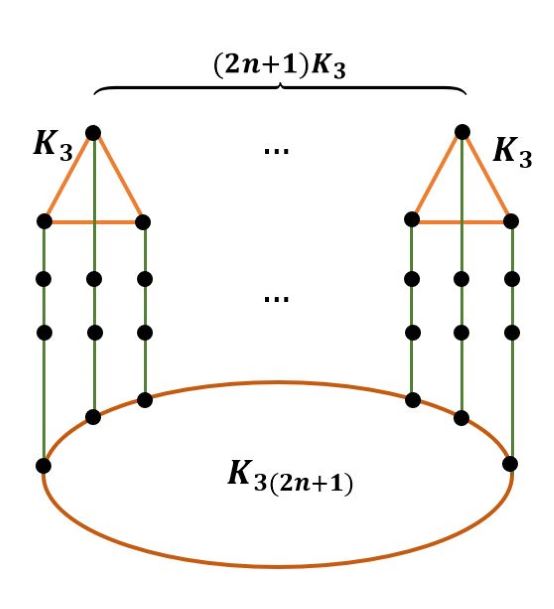}
\caption{The graph $G^*_{n,k}$}
\label{g2}
\end{minipage}
\end{figure}

\noindent{\bf Remark 2.} There are examples showing the sharpness of the conditions in Theorem \ref{t1}.
Let $\ell=1,2$. Let $H_n$ be defined as in Fig. \ref{fh1}.
Suppose that $n\ge k+\ell-2$.
It is easy to see that $H_{n}$ is a $(k+\ell-1)$-connected graph on $5n+2$ vertices and it contains $P_{2\ell}\cup kP_1$ ($P_3\cup kP_1$ as well if $\ell=2$) as an induced subgraph.
From the proof of Theorem 3.3 in  \cite{Chv}, $\tau(H_{n})=\frac{3n}{2n+1}>1$ and $H_{n}$ has no $2$-factor. So $H_{n}$ is a $1$-tough $(k+\ell-1)$-connected graph containing $P_{2\ell}\cup kP_1$ ($P_3\cup kP_1$ as well if $\ell=2$) as an induced subgraph that has no $2$-factor.

\noindent
{\bf Remark 3.}  There are also examples showing the sharpness of the conditions in Theorem \ref{p5}.

For positive integers $k$ and $n$ with $n\ge k-1$, let $G'_{n,k}$ be the graph obtained from $(2n+1)K_{3}\cup K_{3(2n+1)}$ by adding a perfect matching (denoted by $M$) between vertices in $(2n+1)K_{3}$ and vertices in $K_{3(2n+1)}$ and subdividing each edge in $M$ once, see Fig. \ref{g1}. Let $G_{n,k}:=K_n\vee G'_{n,k}$.
Obviously, $G_{n,k}$ is $(k+1)$-connected and contains $P_5\cup kP_1$ as an induced subgraph. As shown in \cite{San}, $\tau(G_{n,k})=2-\frac{n+3}{3(2n+1)+1}>\frac{3}{2}$ and it has no $2$-factor. So $G_{n,k}$ is a $\frac{3}{2}$-tough $(k+1)$-connected graph containing $P_5\cup kP_1$ as an induced subgraph that has no $2$-factor.

For positive integers $k$ and $n$ with $n\ge k$, let $G^*_{n,k}$ be the graph obtained from $(2n+1)K_{3}\cup K_{3(2n+1)}$ by adding a perfect matching (denoted by $M$) between vertices in $(2n+1)K_{3}$ and vertices in $K_{3(2n+1)}$ and subdividing each  edge  in $M$ twice, see Fig. \ref{g2}.  Let $\widehat{G}_{n,k}:=K_n\vee G^*_{n,k}$. Obviously, $\widehat{G}_{n,k}$ is $(k+2)$-connected and contains $P_7\cup kP_1$ ($P_6\cup kP_1$ as well) as an induced subgraph.
For $i=1,\dots,2n+1$, let $u_i$ be one vertex of the $i$th $K_3$ so that $\{u_1,\dots,u_{2n+1}\}$ is independent. Let $u_1v\in M$ and let $u$ the be subdividing vertex on $u_1v$ that is adjacent to $u_1$.
Let \[
W=V(K_n)\cup \left(V((2n+1)K_3)\setminus \{u_1,\dots,u_{2n+1}\}\right)\cup (V(K_{3(2n+1)})\setminus\{v\}).
\]
Then $c(\widehat{G}_{n,k}-W)=3(2n+1)+1$ and
\[
\tau(\widehat{G}_{n,k})=\frac{|W|}{c(\widehat{G}_{n,k}-W)}=2-\frac{n+3}{3(2n+1)+1}>\frac{3}{2}.
\]
Let $A=V(K_n)$ and $B$ be the set of vertices of degree $n+2$ having neighbors in $K_{3(2n+1)}$. Then
$A$ and $B$ are disjoint subsets of vertices of $G=\widehat{G}_{n,k}$ with $|B|=3(2n+1)$. Using the notation of Lemma \ref{tutte} (see below), one gets \[
\delta(A,B)=2n-6(2n+1)+6(2n+1)-(2n+2)=-2.
\]
So $\widehat{G}_{n,k}$ has no $2$-factor by Lemma \ref{tutte}.
Thus $\widehat{G}_{n,k}$ is a $\frac{3}{2}$-tough $(k+2)$-connected graph containing $P_7\cup kP_1$ ($P_6\cup kP_1$ as well) as an induced subgraph that has no $2$-factor.

\section{Preliminaries}

Let $A$ and $B$ be disjoint subsets of vertices of a graph $G$.
For $v\in V(G)$, denote by $e_G(v,B)$ ($e(v,B)$ for simplicity) the number of edges joining $v$ and a vertex of $B$. Let $e(A,B)=\sum_{v\in A}e(v,B)$.
For convenience, we write $e_G(F,B)$ ($e_G(v,F)$ for simplicity) for $e(V(F),B)$ if $F$ is a subgraph of $G$.
The component $H$ of $G-(A\cup B)$ is said to be  odd  (even, respectively) if $e(H,B)$ is odd (even, respectively).
Denote by $o(A,B)$ the number of odd components of $G-(A\cup B)$.
Let \[
\delta(A,B)=2|A|-2|B|+\sum_{v\in B}d_{G-A}(v)-o(A,B).
\]
Then $\delta(A,B)$ is even.
The following lemma due to Tutte gives a criterion for the existence of a $2$-factor.

\begin{lemma}\label{tutte}\cite{Tutte}
A graph $G$ has no $2$-factor if and only if $\delta(A,B)\le -2$ for some disjoint subsets $A,B\subseteq V(G)$.
\end{lemma}

A pair of disjoint subsets $A,B\subseteq V(G)$ with $\delta_G(A,B)\le -2$ is called a barrier  (or Tutte pair \cite{BvS})  of $G$. By Lemma \ref{tutte}, if $G$ has no $2$-factor, then $G$ has a barrier. A biased barrier of $G$ is   a barrier of $G$ such that over all  barriers of $G$, (i) $|A|$ is maximum, and (ii) subject to (i), $|B|$ is minimum.

Let $G$ be a graph with no $2$-factor and $(A,B)$ be a biased barrier of $G$.
For a nonnegative integer $s$, let
\[
\mathcal{C}_s=\{H:\mbox{$H$ is a component of $G-(A\cup B)$ with $e(H,B)=s$}\}.
\]
For $u\in B$, let
\[
o(u)=\left|\left\{H\in \bigcup_{t\ge 1}\mathcal{C}_{2t+1}:e(u,H)=1 \right\}\right|.
\]

Kanno and Shan \cite{KS}  establish the following  fundamental properties of a biased barrier, see also
\cite{GJS}.

\begin{lemma}\label{AB}\cite{KS}
Let $G$ be a graph with no $2$-factor and $(A,B)$ be a biased barrier of $G$. Then
\begin{enumerate}
\item[(i)]
$B$ is independent in $G$.

\item[(ii)]
If $H$ is an even component of $G-(A\cup B)$, then $e(H,B)=0$.

\item[(iii)]
If $H$ is an odd component of $G-(A\cup B)$, then $e(v,H)\le 1$ for any $v\in B$.

\item[(iv)]
If $H$ is an odd component of $G-(A\cup B)$, then $e(v,B)\le 1$ for any $v\in H$.

\item[(v)]
$|B|\ge |A|+\sum_{t\ge 1} t|\mathcal{C}_{2t+1}|+1$.
\end{enumerate}
\end{lemma}

The following lemma is a particular case of \cite[Lemma 10]{GJS}.

\begin{lemma}\label{odd}\cite{GJS}
Let $G$ be a $1$-tough graph on at least three vertices with no $2$-factor and $(A,B)$ be a biased barrier of $G$. Then $\bigcup_{t\ge 1}\mathcal{C}_{2t+1}\ne \emptyset$.
\end{lemma}

Let $P$ by a path from $u$ to $v$. Denote by  $P^{-1}$ the path obtained by reversing $P$.
If $Q$ is a path from $v$ to $w$ and $v$ is the only common vertex of $P$ and $Q$, then
$PQ$ denotes a path
between from $u$  to $w$ passing through $P$ and $Q$.

\section{Proof of Theorem \ref{inde2}}

\begin{lemma}\label{inequality}
Suppose that $x\ge y>0$ and $t$ is a positive integer.
If $a_0,\dots,a_t\ge 2$ with $a_0\le \max_{1\le i\le t}a_i$, then $\frac{x+a_0-2+t}{y+\sum_{i=1}^ta_i-t}\le \frac{x}{y}$. 
\end{lemma}

\begin{proof}
As $2\le a_0\le \max_{1\le i\le t}a_i$ and $0<y\le x$, we have
\[
t(a_0-2)y\le \sum_{i=1}^t(a_i-2)x,
\]
so
\[
(a_0-2+t)y\le t(a_0-2)y+tx\le \left(\sum_{i=1}^ta_i-t\right)x.
\]
That is,
\[
(x+a_0-2+t)y\le \left(y+\sum_{i=1}^ta_i-t\right)x. \qedhere
\]
\end{proof}

\begin{proof}[Proof of Theorem \ref{inde2}]
Suppose by contradictory that $G$ contains no $2$-factor.
Then we have by Lemma \ref{tutte} that $G$ contains a barrier. Let $(A,B)$ be a biased barrier of $G$.
Let $\delta=\delta(G)$ and $\alpha=\alpha(G)$. Then $\delta\ge \epsilon\alpha$, and
 by Lemma \ref{AB}(i), $|B|\le \alpha$.

For $u\in B$, let \[
h_G(u)=\left|\left\{H\in \bigcup_{t\ge 0}\mathcal{C}_{2t+1}:e_G(u,H)=1 \right\}\right|.
\]	
It follows from (i) and (iii) of Lemma \ref{AB} that
\begin{equation}\label{eq0}
|A|\ge \delta-\min_{u\in B}h_G(u)\ge \epsilon\alpha-\min_{u\in B}h_G(u)\ge \epsilon|B|-\min_{u\in B}h_G(u).
\end{equation}
So we have by Lemma \ref{AB}(v) that
\begin{equation}\label{eq}
\sum_{t\ge 1}t|\mathcal{C}_{2t+1}|\le |B|-|A|-1\le 
(1-\epsilon)|B|+\min_{u\in B}h_G(u)-1.
\end{equation}

Suppose first that $\max_{u\in B}h_G(u)\le 1$.
Then $\min_{u\in B}h_G(u)\le 1$.
For each $H\in \mathcal{C}_{2t+1}$ for $t\ge 1$,
Lemma \ref{AB}(iv) ensures that there is a set $W_H$ of arbitrarily chosen $2t$ vertices of $H$ with neighbors in $B$.
Then there is exactly one edge between $V(H)\setminus W_H$ and $B$. Let $W=A\cup \bigcup_{H\in\bigcup_{t\ge 1}\mathcal{C}_{2t+1}}W_H$. Then $W$ is a cut set of $G$ with $c(G-W)\ge |B|$. Hence, by Lemma \ref{AB}(v)  and Eq.~\eqref{eq},
\begin{align*}
\tau(G)&\le \frac{|W|}{c(G-W)}\\
&\le  \frac{|A|+\sum_{t\ge 1}2t|\mathcal{C}_{2t+1}|}{|B|}\\
&\le \frac{|B|+\sum_{t\ge 1}t|\mathcal{C}_{2t+1}|-1}{|B|}\\
&\le  \frac{|B|+(1-\epsilon)|B|+\min_{u\in B}h_G(u)-2}{|B|}\\
&\le\frac{(2-\epsilon)|B|-1}{|B|}\\
&<2-\epsilon,
\end{align*}
contradicting the fact that $G$ is $(2-\epsilon)$-tough.

Suppose next that $\max_{u\in B}h_G(u)\ge 2$.

Let $G_1=G$.
%
Suppose that $G_i$ is defined and $\max_{u\in B} h_{G_i}(u)\ge 1$.
Let $u_i\in B$ be the vertex with $h_{G_i}(u_i)=\max_{u\in B} h_{G_i}(u)$ and let $r_i=h_{G_i}(u_i)$.
Denote by $H_{1}^{(i)},\dots, H_{r_i}^{(i)}$ the components of $G_i-(A\cup B)$ containing a vertex adjacent to $u_i$. 
For $j=1,\dots,r_i$, let \[
W_j^{(i)}=\left\{v\in V(H_j^{(i)}): e(v,B)=1,u_iv\notin E(G)\right\}.
\]
Let $G_{i+1}=G_i-\bigcup_{j=1}^{r_i}V(H_j^{(i)})$.
Repeating this process, we obtain a graph sequence $G_1,\dots, G_s$ with $s\le |B|$ such that
$\max_{u\in B} h_{G_s}(u)=0$.
Let
\[
\ell=\left|\left\{G_i:\max_{u\in B}h_{G_i}(u)\ge 1\right\}\right|
\mbox{ and }
\ell'=\left|\left\{G_i:\max_{u\in B}h_{G_i}(u)\ge 2\right\}\right|.
\]
Then $r_i=1$ if $\ell'+1\le i\le \ell$.
Let
\[
W=A\cup \{u_1,\dots, u_{\ell'}\}\cup \bigcup_{i=1}^{\ell}\bigcup_{j=1}^{r_i}W_j^{(i)}.
\]
From the construction above, any odd component of $G-(A\cup B)$ is of the form $H_j^{(i)}$ for some $i=1,\dots,\ell$ and $j=1,\dots,r_i$.
For each $i=1,\dots,\ell$ and $j=1,\dots,r_i$, $H_j^{(i)}$ is an odd component of $G-(A\cup B)$ by  Lemma \ref{AB}(ii), and $|W_j^{(i)}|=e(H_j^{(i)},B)-1$.
So \[
|W|=|A|+\ell'+\sum_{t\ge 1}2t|\mathcal{C}_{2t+1}|.
\]
As $u_i\in W$ for $i=1,\dots,\ell'$, no two vertices of $B\setminus \{u_i:i=1\dots,\ell'\}$ can be in the same component of $G-W$. 
As there is no edges between $V(H_j^{(i)})\setminus W_j^{(i)}$ and $B\setminus\{u_1,\dots,u_\ell'\}$ for $j=1,\dots,r_i$ and $i=1,\dots,\ell'$ and exactly one edge between $V(H_1^{(i)})\setminus W_1^{(i)}$ and $B\setminus\{u_1,\dots,u_\ell'\}$ if $\ell'+1\le i\le \ell$,
$W$ is a cut set of $G$ with $c(G-W)\ge |B|-\ell'+\sum_{i=1}^{\ell'}h_{G_i}(u_i)$.
Therefore, by Lemma \ref{AB}(v), Eq.~\eqref{eq} again and Lemma \ref{inequality},
\begin{equation}\label{mm}
\begin{aligned}
\tau(G)&\le \frac{|A|+\ell'+\sum_{t\ge  1}2t|\mathcal{C}_{2t+1}|}{|B|-\ell'+\sum_{i=1}^{\ell'}h_{G_i}(u_i)}\\
&\le \frac{|B|+\sum_{t\ge 1}t|\mathcal{C}_{2t+1}|-1+\ell'}{|B|+\sum_{i=1}^{\ell'}h_{G_i}(u_i)-\ell'}\\
&\le \frac{(2-\epsilon)|B|+\min_{u\in B} h_G(u)-2+\ell'}{|B|+\sum_{i=1}^{\ell'}h_{G_i}(u_i)-\ell'}\\
&\le \frac{(2-\epsilon)|B|+\max_{u\in B} h_G(u)-2+\ell'}{|B|+\sum_{i=1}^{\ell'}h_{G_i}(u_i)-\ell'}\\
&\le 2-\epsilon.
\end{aligned}
\end{equation}
Suppose that $\tau(G)= 2-\epsilon$. Then each inequality in \eqref{mm}  is an equality.
As the third inequality is an equality,
Eq.~\eqref{eq} is an equality, so \eqref{eq0} is an equality, implying that
 $|B|=\alpha$. As $\max_{u\in B} h_G(u)\ge 2$, there are two vertices, say $w_1$ and $w_2$ in $G-(A\cup B)$ adjacent to $u_1$.
As $w_1$ and $w_2$ lie in different components of $G-(A\cup B)$ by Lemma \ref{AB}(iii), they are not adjacent. Moreover, they have no neighbors in $B\setminus\{u_1\}$ by Lemma \ref{AB}(iv).
Thus $B\setminus \{u_1\}\cup \{w_1,w_2\}$ is an independent set with $\alpha+1$ vertices, which is a contradiction.
It follows that  $\tau(G)<2-\epsilon$, contradicting the fact that $G$ is $(2-\epsilon)$-tough.
This completes the proof.
\end{proof}

\section{Proof of Theorem \ref{t1}}

\begin{proof}[Proof of Theorem \ref{t1}]
Suppose to the contrary that $G$ contains no $2$-factor. Then we have by Lemma \ref{tutte} that $G$ contains a barrier. Let $(A,B)$ be a biased barrier of $G$.

\begin{claim}\label{c3}
$|B|\ge k+\ell$.	
\end{claim}

\begin{proof}
Suppose first that $\mathcal{C}_1\ne \emptyset$. Assume that $H\in \mathcal{C}_1$. That is, $H$ is a component of $G-(A\cup B)$ with
$e(H,B)=1$. Let $uv$ be the unique edge between $V(H)$ and $B$ with $u\in V(H)$ and $v\in B$.
By Lemma \ref{odd}, $\mathcal{C}_{2t+1}\ne \emptyset$ for some $t\ge 1$, so there exists a component of $G-(A\cup B)$ different from $H$. So $A\cup \{v\}$ is a cut set of $G$.
As $G$ is $(k+\ell-1)$-connected, one gets $|A|\ge k+\ell-2$. So, by Lemma \ref{AB}(v), $|B|\ge |A|+\sum_{t\ge 1}t|\mathcal{C}_{2t+1}|+1\ge |A|+2\ge k+\ell$, as desired.

Suppose next that $\mathcal{C}_1=\emptyset$.
Let $w\in B$. By Lemma \ref{AB}(i)--(iii), $o(w)=d_{G-A}(w)$, so
$A\cup N_{G-A}(w)$ is a cut set of $G$, implying that
$|A|\ge k+\ell-1-o(w)$, as $G$ is $(k+\ell-1)$-connected.
By the definition of $o(w)$, we have \[
o(w)\le \sum_{t\ge 1}|\mathcal{C}_{2t+1}|\le \sum_{t\ge 1}t|\mathcal{C}_{2t+1}|.
\]	
So, by Lemma \ref{AB}(v), \[
|B|\ge |A|+\sum_{t\ge 1}t|\mathcal{C}_{2t+1}|+1\ge k+\ell-1-o(w)+\sum_{t\ge 1}t|\mathcal{C}_{2t+1}|+1\ge k+\ell,
\]
as desired.
\end{proof}

Suppose that $\ell=1$. Let $uv$ be an edge between $V(H)$ and $B$ with $u\in V(H)$ and $v\in B$.
By  (i) and (iv) of Lemma \ref{AB}, $B$ is an independent set and $u$ is not adjacent to any vertex in $B\setminus\{v\}$.
Thus  $G[\{u\}\cup B]\cong P_2\cup (|B|-1)P_1$, so, by Claim \ref{c3},  $G$ contains  $P_2\cup kP_1$ as an induced subgraph, which is a contradiction.

Suppose next that $\ell=2$.
By Lemma \ref{odd}, $\mathcal{C}_{2t+1}\ne \emptyset$ for some $t\ge 1$. So there exists an odd component, say $H'$, in  $G-(A\cup B)$ with $e(H', B)\ge 3$.
By (iii) and (iv) of Lemma \ref{AB}, every vertex of $V(H')$ is adjacent to at most one vertex in $B$ and and vice versa. So, among the vertices of $V(H')$ with  a neighbor in $B$, there are two vertices, say
$u_1$ and $u_2$, such that $u_iv_i\in E(G)$ for $v_i\in B$ with $i=1,2$ and
any internal  vertex on the shortest path  $P$  from $u_1$ to $u_2$  in $H'$ (if any exists) has no neighbors in $B$. Let $P'$ be the path $v_1u_1Pu_2v_2$.
By Lemma \ref{AB}(i), $B$ is an independent set.
So $G[V(P')\cup B]=P'\cup (|B|-2)P_1$.
  Note that $|V(P')|\ge 4$. It then follows from Claim \ref{c3} that $G$ contains $P_4\cup kP_1$ as an induced subgraph, a contradiction. This proves Item (i).

For Item (ii),  as is noted above, if $G$ is $(k+2)$-connected, then $G$ contains an induced $P_4\cup kP_1$, which contains $P_3\cup kP_1$ as an induced subgraph, also a contradiction.
\end{proof}

\section{Proof of Theorem \ref{p5}}

\begin{proof}[Proof of Theorem \ref{p5}]
Suppose to the contrary that $G$ contains no $2$-factor.  Then we have by Lemma \ref{tutte} that $G$ contains a barrier. Let $(A,B)$ be a biased barrier of $G$.

\begin{claim}\label{5c}
There exists some vertex $u\in B$ with $o(u)\ge 2$.
\end{claim}

This is  Claim 4 in the proof of Theorem 13 in  \cite{GJS}, for completeness, however,  we include a proof here.

\begin{proof}
Suppose by contradiction that $o(u)\le 1$ for any $u\in B$.
Let \[
B_0=\{u\in B:o(u)=0\}
\]
and \[
B_1=\{u\in B:o(u)=1\}.
\]
Then $B=B_0\cup B_1$ and $|B_1|=\sum_{t\ge 1}(2t+1)|\mathcal{C}_{2t+1}|$ by (iii) and (iv) of Lemma \ref{AB}.
For each $H\in \mathcal{C}_{2t+1}$ for $t\ge 1$, let $W_H$ be a set of arbitrarily chosen $2t$ vertices having neighbors in $B$. Then there is exactly one edge between $V(H)\setminus W_H$ and  $B$.  Let $W=A\cup \bigcup_{H\in\bigcup_{t\ge 1}\mathcal{C}_{2t+1}}W_H$. As $o(u)\le 1$  for any $u\in B$, no two vertices of $B$ can be in the same component of $G-W$, so $W$ is a cut set of $G$ with
 $c(G-W)\ge |B|$.  As $G$ is $\frac{3}{2}$-tough, we have $|W|\ge \frac{3}{2}|B|$.
Therefore,
\begin{align*}
|A|+\sum_{t\ge 1}2t|\mathcal{C}_{2t+1}|=|W|&\ge \frac{3}{2}|B|\\
&=\frac{3}{2}\left(|B_0|+|B_1|\right)\\
&=\frac{3}{2}\left(|B_0|+\sum_{t\ge 1}(2t+1)|\mathcal{C}_{2t+1}|\right).
\end{align*}
By Lemma \ref{odd}, $\mathcal{C}_{2t+1}\ne \emptyset$ for some $t\ge 1$. Thus, it  follows that \[
|A|+\sum_{t\ge 1}t|\mathcal{C}_{2t+1}|\ge \frac{3}{2}|B_0|+\sum_{t\ge 1}\left(2t+\frac{3}{2}\right)|\mathcal{C}_{2t+1}|>|B_0|+\sum_{t\ge 1}(2t+1)|\mathcal{C}_{2t+1}|=|B|,
\]
a contradiction to Lemma \ref{AB}(v).
\end{proof}

By the same argument as Claim \ref{c3}, we have the following.
\begin{claim}\label{5c2}
$|B|\ge k+\ell$.
\end{claim}

By Claim \ref{5c}, there exists $u\in B$ with
 $o(u)\ge 2$. Thus, there exist odd components $H_1$ and $H_2$  of $G-(A\cup B)$ in $\bigcup_{t\ge 1}\mathcal{C}_{2t+1}$ with $e(u,H_i)=1$ for $i=1,2$.
Let $v_i$ be the neighbor of $u$ in $H_i$ for $i=1,2$.
As $H_1\in \bigcup_{t\ge 1}\mathcal{C}_{2t+1}$, there exists a vertex
 $v_1'\in V(H_1)$ such that $e(v_1',B)=1$ and any internal vertex on the shortest path $P$ in $H_1$ from $v_1$ to $v_1'$ (if any exists) has no neighbors in $B$. Denote by $u_1$ be the unique neighbor of $v_1'$ in $B$.

First, we prove Item (i).

Suppose that $\ell=2$.
By (i) and (iii) of Lemma \ref{AB}, $u$ is  adjacent to neither $u_1$ nor $v_1'$. Let $P'=u_1v_1'P^{-1}v_1uv_2$.
By Lemma \ref{AB}(iii), $N_G(u)\cap(V(H_1)\cup V(H_2))=\{v_1, v_2\}$.
By Lemma \ref{AB}(iv), $N_G(v_1)\cap B=N_G(v_2)\cap B=\{u\}$ and $N_G(v_1')\cap B=\{u_1\}$.
As $B$ is independent by Lemma \ref{AB}(i), one gets $G[B\cup \{v_1,v_1',v_2\}]\cong P'\cup (|B|-2)P_1$.
By Claim \ref{5c2}, $G$ contains $P_5\cup kP_1$ as an induced subgraph, a contradiction.

Suppose next that $\ell=3$.
By argument as above, there exists a vertex $v_2'\in V(H_2)$ such that $e(v_2',B)=1$ and any internal vertex on the shortest path $Q$ in $H_2$ from $v_2$ and $v_2'$ (if any exists) has no neighbors in $B$.
Denote by $u_2$ the unique neighbor of $v_2'$ in $B$.
It then follows from Lemma \ref{AB}(iii) that $u$ is adjacent to neither $v_1'$ nor $v_2'$.

\noindent{\bf Case 1.} $u_1\ne u_2$.

As above, by (iii) and (iv) of Lemma \ref{AB},  $u$ has only neighbors $v_1$ and $v_2$ in $V(H_1)\cup V(H_2)$, $u_i$ has only neighbor $v_i'$ in
in $V(H_1)\cup V(H_2)$ for $i=1,2$, $v_1$ and $v_2$ have only neighbor $u$ in $B$ and $v_i'$ has only neighbor $u_i$ in $B$ for $i=1,2$. Recall that $B$ is independent by Lemma \ref{AB}(i). Thus
 $G[B\cup V(P)\cup V(Q)]\cong L\cup (|B|-3)P_1$, where $L=u_1v_1'P^{-1}v_1uv_2Qv_2'u_2$ is an induced path of $G$. Evidently, $|V(L)|\ge 7$. By Claim \ref{5c2}, $G$ contains $P_7\cup kP_1$ as an induced subgraph, a contradiction.

\noindent{\bf Case 2.} $u_1=u_2$.

As $e(H_2,B)\ge 3$, there exists a vertex $v_2^*$ in $H_2\setminus \{v_2, v_2'\}$ having a neighbor in $B$ with the smallest distance to $v_2$ or $v_2'$ in $H_2$.
Denote by $u_2^*$ the unique neighbor of $v_2^*$ in $B$.
By Lemma \ref{AB}(iii) again, $u_2^*$ is not adjacent to any of $v_1,v_1',v_2,v_2'$. By Lemma \ref{AB}(i), $B$ is independent.
If $v_2'$ lies outside the shortest path $R$ from $v_2$ to $v_2^*$, $G[B\cup V(P)\cup V(R)]\cong L^*\cup (|B|-3)P_1$, where
$L^*=u_2^*v_2^*R^{-1}v_2uv_1Pv_1'u_1$ is an induced path of $G$.
By Claim \ref{5c2}, $G$ contains $P_7\cup kP_1$ as an induced subgraph, a contradiction.
Thus  $v_2'$ lies on the shortest path  from $v_2$ to $v_2^*$. Let $R'$ be a shortest path from $v_2'$ to $v_2^*$. Then $G[B\cup V(P)\cup V(R')]=L^{**}\cup (|B|-3)P_1$,
where $L^{**}=u_2^*v_2^*R'v_2'u_1v_1'P^{-1}v_1u$ is an induced path of $G$, so $G$ contains  $P_7\cup kP_1$ as an induced subgraph, also a contradiction.
This completes the proof of Item (i).

Next, we prove Item (ii). As is stated above, $G$ contains $P_7\cup kP_1$ as an induced subgraph, so it contains $P_6\cup kP_1$, a contradiction. This proves Item (ii).
\end{proof}

\vspace{5mm}

\noindent {\bf Acknowledgement.}
This work was supported by the National Natural Science Foundation of China (No.~12071158).


\begin{thebibliography}{99}

\bibitem{BKKV} D. Bauer, G.Y. Katona, D. Kratsch, H.J. Veldman,
Chordality and $2$-factors in tough graphs,
Discrete Appl. Math. 99 (2000) 323--329.

\bibitem{BvS} D. Bauer, J. van den Heuvel, E. Schmeichel,
$2$-Factors in triangle-free graphs, J. Graph Theory  21 (4) (1996) 405--412.

\bibitem{BS} D. Bauer,  E. Schmeichel,
Toughness, minimum degree,  and the existence of $2$-factors,  J. Graph Theory 18 (3) (1994) 241--256.


\bibitem{Chv} V. Chv\'{a}tal,
Tough graphs and Hamiltonian circuits,
Discrete Math. 5 (1973) 215--228.

\bibitem{CE} V. Chv\'{a}tal, Erd\H{o}s,
A note on Hamiltonian circuits, Discrete Math. 2 (1972) 215--228.

\bibitem{EJKS} H. Enomoto, B. Jackson, P. Katerinis, A. Saito,
Toughness and the existence of $k$-factors,
J. Graph Theory 9 (1) (1985) 87--95.

\bibitem{GJS} E. Grimm, A. Johnsen, S. Shan,
Existence of $2$-factors in tough graphs without forbidden subgraphs,
Discrete Math. 346 (2023) 113578.

\bibitem{PRP} P. Holub, Z.  Ryj\'{a}\v{c}ek, P. Vr\'{a}na, S. Wang, L. Xiong,
Forbidden pairs of disconnected graphs for $2$-factor of connected graphs,
J. Graph Theory 100 (2) (2022) 209--231.

\bibitem{KS} J. Kanno, S. Shan,
Vizing's $2$-factor conjecture involving toughness and maximum degree conditions,
Electron. J. Combin. 26 (2) (2019) Paper 2.17.

\bibitem{OS1} K. Ota, M. Sanka,
Hamiltonian cycles in $2$-tough $2K_2$-free graphs,
J. Graph Theory 101 (4) (2022) 769--781.

\bibitem{Nis} T. Niessen,
Minimum degree, independence number and regular factors, Graph Combin. 11 (1995) 367--378.



\bibitem{San} M. Sanka,
Forbidden subgraphs and $2$-factors in $3/2$-tough graphs,
J. Graph Theory 103 (2) (2023) 271--284.

\bibitem{SS} L. Shi, S. Shan, A note on hamiltonian cycles in $4$-tough $(P_2\cup kP_1)$-free
graphs, Discrete Math. 345 (2022) 113081.

\bibitem{Tutte} W.T. Tutte,
A short proof of the factor theorem for finite graphs,
Can. J. Math. 6 (1954) 347--352.

\bibitem{XLZ} L. Xu, C. Li, B. Zhou,
Hamiltonicity of $1$-tough $(P_2\cup kP_1)$-free graphs, Discrete Math., ArXiv:2303.09741v2.


\end{thebibliography}
\end{document}